\documentclass[a4paper,12pt]{amsart}
\usepackage{amsmath, amssymb, mathrsfs}
\usepackage[dvipdfmx]{hyperref}
\hypersetup{colorlinks=true,linkcolor=blue,citecolor=blue}

\numberwithin{equation}{section}
\newtheorem{thm}[equation]{Theorem}
\newtheorem{prop}[equation]{Proposition}
\newtheorem{lem}[equation]{Lemma}
\newtheorem{cor}[equation]{Corollary}
\newtheorem{clm}[equation]{Claim}

\theoremstyle{definition}
\newtheorem{defn}[equation]{Definition}

\theoremstyle{remark}
\newtheorem{rem}[equation]{Remark}
\newtheorem{ex}[equation]{Example}

\newcommand{\N}{\ensuremath{\mathbb{N}}}

\newcommand{\R}{\ensuremath{\mathbb{R}}}
\newcommand{\tX}{{\tilde{X}}}
\newcommand{\tY}{{\tilde{Y}}}
\newcommand{\muw}{\mu_{W}}
\newcommand{\mux}{\mu_{X}}
\newcommand{\muy}{\mu_{Y}}
\newcommand{\muz}{\mu_{Z}}

\newcommand{\seqk}[1]{ \{#1\}_{k=1}^{\infty} }
\newcommand{\seql}[1]{ \{#1\}_{l=1}^{\infty} }

\newcommand{\seqn}[1]{ \{#1\}_{n=1}^{\infty} }
\DeclareMathOperator{\diam}{diam}
\DeclareMathOperator{\pt}{pt}

\newcommand{\niN}{{n\in\N}}
\newcommand{\nti}{{n\to\infty}}

\newcommand{\ep}{\varepsilon}
\renewcommand{\phi}{\varphi}

\newcommand{\cA}{{\mathcal{A}}}

\newcommand{\cD}{\mathcal{D}}
\newcommand{\cH}{\mathcal{H}}
\newcommand{\cK}{\mathcal{K}}
\newcommand{\cL}{\mathcal{L}}
\newcommand{\cM}{\mathcal{M}}
\newcommand{\cP}{\mathcal{P}}
\newcommand{\cQ}{\mathcal{Q}}
\newcommand{\cX}{\mathcal{X}}
\newcommand{\cY}{\mathcal{Y}}
\newcommand{\cZ}{\mathcal{Z}}
\newcommand{\haus}{d_{\mathrm{H}}}
\newcommand{\proh}{d_{\mathrm{P}}}
\newcommand{\GH}{d_{\mathrm{GH}}}

\DeclareMathOperator{\supp}{supp}

\title[Boundedness of sets of metric measure spaces in pyramids]{Boundedness of measured Gromov--Hausdorff precompact sets of metric measure spaces in pyramids}
\author[D. Kazukawa]{Daisuke Kazukawa}
\address{Faculty of Mathematics, Kyushu University, Fukuoka 819-0395, JAPAN}
\email{kazukawa@math.kyushu-u.ac.jp}
\author[T. Yokota]{Takumi Yokota}
\address{Mathematical Institute, Tohoku University, Sendai 980-8578, JAPAN}
\email{takumiy@tohoku.ac.jp}

\date{Octorber 1, 2022}

\begin{document}
\maketitle
\begin{abstract}
We prove that any measured Gromov--Hausdorff precompact set of metric measure spaces which is contained in a certain set, called a pyramid,
is bounded by some metric measure space with respect to the Lipschitz order inside the pyramid.
This is proved as a step towards a possible extension of the statement of Gromov, for which we gave a detailed proof in our previous work.
Several related results are also obtained.
\end{abstract}

\section{Introduction}

This is a sequel to our previous work \cite{KY},
where we studied \emph{metric measure spaces}, or \emph{mm-spaces} for short.
An mm-space is a triple $X=(X, d_X, \mux)$,
where $(X, d_X)$ is a complete separable metric space with a Borel probability measure $\mux$ on $X$.
The set $\cX$ of all (isomorphism classes of) mm-spaces has a topology induced by the distance function $\Box$, called the \emph{box distance},
and a partial order $\prec$, called the \emph{Lipschitz order}, both of which were introduced by Gromov~\cite{G}.
Our main reference is Shioya's monograph~\cite{S}.

The box distance could be thought of as a measured analog of the GH distance between compact metric spaces.
Here and hereafter, GH stands for Gromov--Hausdorff.
We refer to e.g.~Tuzhilin~\cite{T} for the GH distance.

In this paper, our interest is in certain subsets of $\cX$, called \emph{pyramids},
which were also introduced by Gromov~\cite{G}
as elements of a natural compactification of $\cX$ topologized by the so-called observable distance,
e.g.~\cite[Theorems~6.12 and 6.23]{S}.
Necessary definitions are recalled in Section~\ref{sec;pre}.
Our main result is the following.
\begin{thm}\label{thm;main}
If $\cP$ is a pyramid and $\cY\subset\cP$ is a subset such that the support $\supp\muy\subset Y$ of $\muy$ is compact for any $Y\in\cY$ and the set
$\{\supp\muy:Y\in\cY\}$
is GH-precompact,
then there exists a compact mm-space $X\in\cP$ with $Y\prec X$ for any $Y\in\cY$.
\end{thm}

There is also the notion of \emph{measured GH (mGH) convergence} for sequences of compact mm-spaces, e.g.~\cite[Definition~4.33]{S}.
The relation between the assumption of Theorem~\ref{thm;main} and the mGH-precompactness is as follows:
If $\seqn{Y_n}$ is a sequence in $\cY$ in Theorem~\ref{thm;main},
then a subsequence of $\seqn{\supp\mu_{Y_n}}$ mGH-converges to some compact mm-space $Y$ with possibly $\supp\muy\ne Y$.
If a set $\cY$ of mm-spaces is mGH-precompact, then the sets $\cY$ and $\{\supp\muy:Y\in\cY\}$ are GH-precompact,
cf.~\cite[Remark~4.34]{S}.

If $\cY$ in Theorem~\ref{thm;main} is a finite subset,
then the existence of such $X\in\cP$ follows easily from Lemma~\ref{lem;cptdominate}.

As $\cX$ is itself a pyramid,
our Theorem~\ref{thm;main} was inspired by the following result,
whose proof was sketched in \cite{G} and detailed in \cite{KY}.
\begin{thm}[Gromov~{\cite[3$\frac{1}{2}$.15.(f)]{G}}, cf.~\cite{KY}]\label{thm;G}
If a subset $\cY\subset\cX$ is $\Box$-precompact,
then there exists an mm-space $X$ with $Y\prec X$ for any $Y\in\cY$.
\end{thm}

We know that the set $\cY$ in Theorem~\ref{thm;main} is $\Box$-precompact by Lemmas~\ref{lem;precpt} and \ref{lem;GHprecpt}.
We do not know to what extent Theorems~\ref{thm;main} and \ref{thm;G} could be generalized.

In Example~\ref{ex;XDN},
we construct examples of $\cY$ and $X$ as in Theorem~\ref{thm;G}.

The following is an interesting corollary of Theorem~\ref{thm;main}.
\begin{cor}\label{cor;main}
Let $\cY$ be a set of essentially nonbranching mm-spaces which satisfies the ${\rm CD}^*(K, N)$ condition
and $\cP$ the set of all (isomorphism classes of) mm-spaces satisfying the Talagrand inequality $T_2 (KN/(N-1))$
for some $K>0$ and $N\in (1, \infty)$.
Then there exists a compact mm-space $X\in\cP$ with $Y\prec X$ for any $Y\in\cY$.
\end{cor}

In Corollary~\ref{cor;main},
the set $\cY$  is GH-precompact by the Myers theorem and the Bishop--Gromov inequality,
the fact that $\cP$ is a pyramid follows from the combination of e.g.~\cite[Proposition~5.5]{S},
Bakry--Gentil--Ledoux~\cite[Proposition~9.2.4]{BGL}, and Ozawa--Suzuki~\cite[Theorem~1.1]{OS},
cf.~\cite[Lemma~2.1, Proposition~3.20]{K;CAG},
and that $\cY\subset\cP$ follows from Cavalletti--Mondino~\cite[Theorem~1.6]{CM}.
Then Corollary~\ref{cor;main} is a direct consequence of Theorem~\ref{thm;main}.

The ${\rm CD}^*(K, N)$ condition is the so-called reduced curvature-dimension condition
and the Talagrand inequality is one of the functional inequalities that are studied extensively on mm-spaces.
We refer to the above references for the definitions.

We defined a natural analog of the Lipschitz order $\prec$ for metric spaces in \cite{KY}
and we will define \emph{pyramids of compact metric spaces} in Definition~\ref{def;GHpyramid},
cf.~Nakajima--Shioya~\cite{NS} and Remark~\ref{rem;NS}.
Then the following is a GH analog of Theorem~\ref{thm;main} and extends \cite[Proposition~1.4]{KY},
which deals with the case with $\cP$ being the set of all isometry classes of compact metric spaces.
\begin{prop}\label{prop;main}
If $\cP$ is a pyramid of compact metric spaces and $\cY\subset\cP$ is a GH-precompact subset,
then there exists a compact metric space $X\in\cP$ with $Y\prec X$ for any $Y\in\cY$.
\end{prop}

For a complete separable metric space $Z$,
we let $\cP(Z)$ denote the set of all Borel probability measures on $Z$ equipped with the Prohorov distance $\proh$.
Prohorov's theorem states that a subset $\cM\subset\cP(Z)$ is precompact in $(\cP(Z), \proh)$ if and only if $\cM$ is tight.

The following is a byproduct of this work,
which gives another characterization of $\Box$-precompact subsets in $\cX$.
A similar statement is well known for GH-precompact sets of compact metric spaces, e.g.~\cite[Theorem~7.19]{T}.
\begin{prop}\label{prop;embed}
Let $\cY$ be a set of mm-spaces.
Then the following are equivalent.
\begin{enumerate}
\item $\cY$ is $\Box$-precompact.
\item There exists a complete separable metric space $Z$ for which
any $Y\in\cY$ admits an isometric embedding $f_Y :\supp\muy\to Z$
so that $\{(f_Y)_* \muy\}_{Y\in\cY} \subset\cP(Z)$ is tight.
\end{enumerate}
\end{prop}

Gigli--Mondino--Savar\'{e}~\cite{GMS} introduced the notion of \emph{pointed measured Gromov (pmG) convergence}
for sequences of pointed metric spaces with locally finite measures.
Its definition is recalled in Definition~\ref{def;pmG}.
The following states that it is essentially equivalent to the $\Box$-convergence in our setting.
This was proved in the thesis \cite{K;thesis} of the first author, but here we prove it as a corollary of Lemma~\ref{lem;countableembed}.
\begin{cor}[\cite{K;thesis}]\label{cor;pmG}
Let $X_n$ for $\niN$ and $X$ be mm-spaces.
Then the following are equivalent.
\begin{enumerate}
\item $\seqn{X_n}$ $\Box$-converges to $X$.
\item There exist $x_n\in\supp\mu_{X_n}$ and $x\in\supp\mux$ with which $\seqn{(X_n, x_n)}$ pmG-converges to $(X, x)$.
\item For any $x\in\supp\mux$, there exist $x_n \in\supp\mu_{X_n}$ with which $\seqn{(X_n, x_n)}$ pmG-converges to $(X, x)$.
\end{enumerate}
\end{cor}

After some preparations in Sections~\ref{sec;pre} and \ref{sec;lem},
we present proofs of Theorem~\ref{thm;main} and Proposition~\ref{prop;main} in Section~\ref{sec;pf}.
Several results related to our Theorem~\ref{thm;main} are collected in Section~\ref{sec;mis}.
Proofs of Proposition~\ref{prop;embed} and Corollary~\ref{cor;pmG} are given in Section~\ref{sec;embed}.

\section{Preliminaries}\label{sec;pre}
In this section, we recall some definitions and facts from \cite{S} and \cite{KY} and define pyramids of compact metric spaces in Definition~\ref{def;GHpyramid}.
This section could be safely skipped possibly except for Definition~\ref{def;GHpyramid}.

We will use the notations in \cite{S} such as
\[
U_\ep (A) := \{ x\in X : d_X (x, A) <\ep \}, \qquad
B_\ep (A) := \{ x\in X : d_X (x, A) \le\ep \}
\]
for a subset $A\subset X$ of a metric space $(X, d_X)$ and $\ep>0$.

Let $\cP(X)$ denote the set of all Borel probability measures on a metric space $X$.
First, we collect useful facts on the Prohorov distance $\proh$ on $\cP(X)$.
\begin{lem}[e.g. {\cite[Lemmas 1.26, 4.36]{S}}]\label{lem;P}
Let $X$ and $Y$ be metric spaces.
Suppose that $f, g:X\to Y$ are Borel maps, $\mu, \nu\in\cP(X)$, $\ep\ge 0$, and $\tX\subset X$ is a Borel set with $\mu(\tX)\ge 1-\ep$.
\begin{enumerate}
\item If $d_Y (f(x), g(x))\le\ep$ for any $x \in\tX$, then $\proh(f_* \mu, g_* \mu)\le \ep$.
\item If $d_Y (f(x), f(x'))\le d_X (x, x')+\ep$ for any $x, x'\in\tX$ and $\nu(\tX)\ge 1-\ep$, then $\proh(f_* \mu, f_* \nu) \le \proh(\mu, \nu)+2\ep$.
\end{enumerate}
\end{lem}

The \emph{support} $\supp\mu$ of a Borel measure $\mu$ on a separable metric space $X$ is the set of points $x\in X$ with
$\mu(O)>0$ for any open set $O\subset X$ with $x\in O$.
We say that two mm-spaces $X$ and $Y$ are \emph{mm-isomorphic}
if there exists an isometry $f:\supp\mux\to\supp\muy$ with $f_* \mux=\muy$.
We let $\cX$ denote the set of all mm-isomorphism classes of mm-spaces,
but we do not distinguish an mm-space and its mm-isomorphism class in $\cX$.
Since any mm-space $X$ is mm-isomorphic to $(\supp\mux, d_X, \mux)$,
hereafter an mm-space $X$ is supposed to satisfy $X=\supp\mux$ as usual except for a few exceptions.

\begin{defn}[e.g.~{\cite[Definition~4.4]{S}}]
Let $\cL$ denote the $1$-dimensional Lebesgue measure on the closed unit interval $I:=[0, 1]\subset\R$.
A Borel map $\phi:I\to X$ is called a \emph{parameter} of an mm-space $X$ if it satisfies $\phi_* \cL=\mux$.
Any mm-space admits a parameter, e.g.~\cite[Lemma~4.2]{S}.

The \emph{box distance} $\Box(X, Y)$ between mm-spaces $X$ and $Y$ is defined as the infimum of $\ep>0$ for which
there exist parameters $\phi:I\to X$ and $\psi:I\to Y$ and a Borel set $I'\subset I$ with $\cL(I')>1-\ep$ and
\[
|d_X (\phi(s), \phi(t)) -d_Y (\psi(s), \psi(t))|<\ep
\]
for any $s, t\in I'$.
\end{defn}

We know that $(\cX, \Box)$ is a metric space, e.g.~\cite[Theorem~4.10]{S}.

The following gives a relation between the box and Prohorov distances.
\begin{lem}[e.g.~{\cite[Proposition~4.12]{S}}]\label{lem;BP}
Let $X$ be a complete separable metric space and $\mu, \nu\in\cP(X)$.
Then we have
\[
\Box((X, \mu), (X, \nu))\le 2\proh(\mu, \nu).
\]
\end{lem}

The following characterizes $\Box$-precompact sets of mm-spaces, cf.~\cite[Corollary~3.20]{KY}.
\begin{lem}[e.g.~{\cite[Lemma~4.28]{S}}]\label{lem;precpt}
Let $\cY\subset\cX$.
Then $\cY$ is $\Box$-precompact if and only if
there exists $\Delta=\Delta(\ep)<\infty$ for any $\ep>0$ such that
any $Y\in\cY$ admits a set $\cK_Y$ of Borel sets in $Y$ such that
\[
\#\cK_Y \le\Delta, \quad
\max_{K\in\cK_Y} \diam K \le\ep, \quad
\diam \bigcup\cK_Y \le\Delta, \quad\text{ and }\quad
\muy(\bigcup\cK_Y)\ge 1-\ep.
\]
\end{lem}

We also use a GH counterpart of Lemma~\ref{lem;precpt} in the following form.
\begin{lem}[cf. {\cite[Theorem 7.18]{T}}]\label{lem;GHprecpt}
Let $\cY$ be a set of compact metric spaces.
Then $\cY$ is GH-precompact if and only if $\sup_{Y\in\cY} \diam Y<\infty$ and
there exists $\Delta=\Delta(\ep)<\infty$ for any $\ep>0$ such that
any $Y\in\cY$ admits a set $\cK_Y$ of disjoint Borel sets of $Y$ such that
\[
\#\cK_Y \le\Delta, \quad
\max_{K\in\cK_Y} \diam K \le\ep, \quad\text{ and }\quad
\bigsqcup\cK_Y =Y.
\]
\end{lem}

\begin{defn}[e.g.~\cite{S}, \cite{KY}]\label{def;precep}
Let $X$ and $Y$ be mm-spaces, $\cY\subset\cX$ a subset, and $\ep\ge 0$.

We say that $X$ \emph{dominates} $Y$ and write $Y\prec X$ if there exists a $1$-Lipschitz map $f:X\to Y$ with $f_* \mux=\muy$.

We write $\cY\prec X$ if $Y\prec X$ for any $Y\in\cY$.

We write $Y\prec_\ep X$ if
there exists a Borel set $\tX\subset X$, called a \emph{nonexceptional domain}, and
a Borel map $f:X\to Y$ such that
\begin{enumerate}
\item $\mux(\tX)\ge 1-\ep$,
\item $d_Y (f(x), f(x')) \le d_X (x, x') +\ep$ for any $x, x'\in \tX$, and
\item $\proh(f_* \mux, \muy) \le\ep$.
\end{enumerate}
\end{defn}

\begin{rem}\label{rem;prec}
We know that $\prec$ is a partial order on $\cX$ by e.g. \cite[Proposition~2.11]{S}.

Let $X$ and $Y$ be mm-spaces.
We know that $X, Y\prec X\times Y$, where the product $X\times Y$ is equipped with a product metric and the product measure,
e.g.~\cite[Paragraph below Definition~6.3]{S}.
If $\Box(X, Y)<\ep$, then $Y\prec_{3\ep} X$ by e.g.~\cite[Lemma 4.22]{S}.
The following conditions are equivalent by e.g. Lemma~\ref{lem;precep}:
\begin{enumerate}
\item $Y\prec X$.
\item $Y\prec_0 X$.
\item $Y\prec_\ep X$ for all $\ep>0$.
\end{enumerate}
\end{rem}

The following two lemmas involve the box distance $\Box$ and the Lipschitz order $\prec$.
\begin{lem}[{\cite[Lemma~6.10]{S}}]\label{lem;dominatedsubconv}
Let $\seqn{Y_n}$, $\seqn{Z_n}$, and $\seqn{W_n}$ be sequences of mm-spaces.
Suppose that $\seqn{Y_n}$ and $\seqn{Z_n}$ $\Box$-converge to mm-spaces $Y$ and $Z$ respectively and that $Y_n, Z_n \prec W_n$ for any $n$.
Then there exists a sequence $\seqn{X_n}$ of mm-spaces which satisfies $Y_n, Z_n \prec X_n \prec W_n$ for any $n$ and has a $\Box$-convergent subsequence.
\end{lem}

\begin{lem}[cf.~{\cite[Lemma~4.39]{S}, \cite[Proposition~3.15]{KY}}]\label{lem;precep}
Let $\seqn{\ep_n}$ be a sequence of positive numbers with $\ep_n \to 0$ as $\nti$.
If sequences $\seqn{X_n}$ and $\seqn{Y_n}$ of mm-spaces $\Box$-converge to mm-spaces $X$ and $Y$ respectively and $Y_n \prec_{\ep_n} X_n$ for any $n$,
then $Y\prec X$.
\end{lem}

\begin{defn}[{\cite[Definitions~6.3, 6.4]{S}}]\label{def;pyramid}
A \emph{pyramid} is a nonempty closed set $\cP$ in $(\cX, \Box)$ with the following properties:
\begin{enumerate}
\item If $X\in\cP$ and $Y\in\cX$ satisfy $Y\prec X$, then $Y\in\cP$.
\item If $Y, Z\in\cP$, then there exists $X\in\cP$ with $Y, Z\prec X$.
\end{enumerate}

We say that a sequence $\seqn{\cP_n}$ of pyramids \emph{converges weakly} to a pyramid $\cP$ if
\begin{enumerate}
\item $\lim_{n\to\infty} \Box(X, \cP_n)=0$ for any $X\in\cP$ and
\item $\liminf_{n\to\infty} \Box(Y, \cP_n)>0$ for any $Y\in\cX\setminus\cP$.
\end{enumerate}
\end{defn}

Typical examples of pyramids are $\cX$ and
\[
\cP_X := \{Y\in\cX:Y\prec X\}
\]
for an mm-space $X$.

\begin{thm}[{\cite[Theorem~6.22]{S}}]\label{thm;rho}
There exists a metric $\rho$ which makes the set of all pyramids $\Pi$ a compact metric space $(\Pi, \rho)$ and
which is compatible the weak convergence of pyramids.
\end{thm}

Following \cite{NS},
we let $\cH$ denote the set of all isometry classes of compact metric spaces and equip it with the GH distance $\GH$,
but we do not distinguish a compact metric space and its isometry class in $\cH$.

\begin{defn}[{\cite[Definition 4.3]{KY}}, cf.~\cite{NS}]\label{def;GHprec}
Let $X$ and $Y$ be metric spaces and $\ep\ge 0$.

We write $Y\prec X$ if there exists a $1$-Lipschitz surjection $f:X\to Y$.

We write $Y\prec_\ep X$ if there exists a map $f:X\to Y$ such that
\begin{enumerate}
\item $d_Y (f(x), f(x')) \le d_X (x, x') +\ep$ for any $x, x' \in X$ and
\item $d_Y (y, f(X)) \le\ep$ for any $y\in Y$.
\end{enumerate}
\end{defn}

The following is a GH analog of Lemma~\ref{lem;precep}.
\begin{lem}[e.g. {\cite[Lemma~4.5]{KY}}]\label{lem;GHprecep}
Let $\seqn{\ep_n}$ be a sequence of positive numbers with $\ep_n \to 0$ as $\nti$.
If sequences $\seqn{X_n}$ and $\seqn{Y_n}$ of compact metric spaces GH-converge to compact metric spaces $X$ and $Y$ respectively and $Y_n \prec_{\ep_n} X_n$ for any $n$,
then $Y\prec X$.
\end{lem}

\begin{defn}[cf. \cite{NS}]\label{def;GHpyramid}
We call a nonempty closed set $\cP$ in $(\cH, \GH)$ a \emph{pyramid} (of compact metric spaces) if
\begin{enumerate}
\item $Y\in\cP$ provided that $X\in\cP$ and $Y\prec X$ and
\item there exists $X\in\cP$ with $Y, Z\prec X$ provided that $Y, Z\in\cP$.
\end{enumerate}
\end{defn}

By Lemma~\ref{lem;GHprecep}, the set
\[
\cP_X :=\{Y\in\cH:Y\prec X\}
\]
is a pyramid for any $X\in\cH$.

\begin{rem}\label{rem;NS}
Nakajima--Shioya~\cite{NS} introduced a relation $\precsim$ between extended metric spaces
and defined pyramids of compact metric spaces which are different from ours in Definition~\ref{def;GHpyramid}.
For compact metric spaces $X$ and $Y$,
we know that $Y\precsim X$ if and only if there exists a map $f:Y\to X$ which is expanding, i.e.,
\[
d_Y (y, y') \le d_X (f(y), f(y'))
\]
for any $y, y' \in Y$, cf.~\cite[Lemma~3.1]{NS}.

We note that
\begin{itemize}
\item $Y\prec X$ implies $Y\precsim X$ and
\item $Y\precsim X$ implies that there exists a subset $X' \subset X$ with $Y\prec X'$
\end{itemize}
for compact metric spaces $X$ and $Y$.
\end{rem}

\section{Auxiliary lemmas}\label{sec;lem}
In this section, we prove several lemmas which we need in the following sections.
We equip the set $\R^\infty$ with the sup norm $\|\cdot\|_\infty$.

\begin{lem}\label{lem;abcd}
Let $a, b, c, d\in\R$ and $\ep>0$.
Suppose that $0\le a-c\le\ep$ and $0\le b-d\le\ep$.
Then
\[
| |a-b|-|c-d| |\le\ep.
\]
\end{lem}

\begin{proof}
We may assume that $a\ge b$.

If $a-c\le b-d$, then $0\le a-b\le c-d$ and
\[
| |a-b|-|c-d| | = (c-d)-(a-b) = (b-d)-(a-c) \le b-d \le \ep.
\]

If $a-c\ge b-d$ and $c\ge d$, then $a-b\ge c-d\ge 0$ and
\[
| |a-b|-|c-d| | = (a-b)-(c-d) = (a-c)-(b-d) \le a-c\le\ep.
\]

If $c<d$, then $a-d\ge a-b\ge 0$, $b-c\ge d-c>0$, and
\[
| |a-b|-|c-d| | = |(a-d)-(b-c)| \le \max\{ a-d, b-c \} \le a-c\le\ep.
\]

Now the proof is complete.
\end{proof}

\begin{defn}
For an mm-space $X$, a complete metric space $Y$, and a Borel map $f:X\to Y$,
we define an mm-space $f_* X :=(\supp f_* \mux, d_Y, f_* \mux)$.
\end{defn}

\begin{rem}\label{rem;f*X}
Suppose that $X$ is an mm-space, $Y$ and  $Z$ are complete metric spaces, $f:X\to Y$ and $g:X\to Z$ are Borel maps, and $\ep>0$.

If $f$ is $1$-Lipschitz, then $f_* X \prec X$.

If $\tX\subset X$ is a Borel set with $\mux(\tX)\ge 1-\ep$ and
\[
|d_Y (f(x), f(x'))-d_Z (g(x), g(x'))|\le\ep
\]
for any $x, x' \in\tX$, then $\Box(f_* X, g_* X)\le\ep$.
\end{rem}

\begin{lem}[cf. {\cite[Lemma~4.6]{K}}]\label{lem;K}
Let $\ep>0$.
If $X$ and $Y$ are mm-spaces with $Y \prec_\ep X$,
then there exists an mm-space $Z$ with $Z\prec X$ and $\Box(Y, Z) \le 3\ep$.
\end{lem}

\begin{proof}
By assumption, there exist a Borel map $f:X\to Y$ and a Borel set $\tX\subset X$ with $\mux(\tX)\ge 1-\ep$,
\begin{align}\label{ineq;f}
d_Y (f(x), f(x')) \le d_X (x, x')+\ep
\end{align}
for any $x, x' \in\tX$, and $\proh(f_* \mux, \muy)\le\ep$.
Let $\iota:Y\to\R^\infty$ be an isometric embedding, e.g. the Kuratowski embedding.
Define functions $f_n, g_n:X\to\R$ by
\[
(f_n (x))_{n=1}^\infty = \iota\circ f(x) \quad\text{ and }\quad
g_n (x):= \inf_\tX [ f_n (\cdot)+d_X (x, \cdot) ].
\]
Then $g_n$ is $1$-Lipschitz and satisfies $0\le f_n (x) -g_n (x)\le\ep$, cf.~\cite[Proof of Lemma~5.4]{S}, and
\begin{align*}
\left|\, |f_n (x)-f_n (x')| - |g_n (x)-g_n (x')| \,\right|\le\ep
\end{align*}
for any $x, x'\in\tX$ by Lemma~\ref{lem;abcd}.
Hence the $1$-Lipschitz map $g:=(g_n)_{n=1}^\infty :X\to\R^\infty$ satisfies
\begin{align*}
\left|\, \|\iota\circ f(x)-\iota\circ f(x')\|_\infty - \|g(x)-g(x')\|_\infty \,\right|\le\ep,
\end{align*}
for any $x, x'\in\tX$.

Then $Z:=g_* X$ is an mm-space with $Z\prec X$ and
\[
\Box(Y, Z)
\le \Box(Y, f_* X) +\Box((\iota\circ f)_* X, Z)
\le 2\proh(\muy, f_* \mux) +\ep
\le 3\ep
\]
by Lemma~\ref{lem;BP} and Remark~\ref{rem;f*X}.
\end{proof}

The following is a GH analog of Lemma~\ref{lem;K}.
\begin{lem}[cf.~{\cite[Lemma~3.5]{NS}}]\label{lem;KGH}
Let $\ep>0$.
If $X$ and $Y$ are compact metric spaces with $Y\prec_\ep X$,
then there exists a compact metric space $Z$ with $Z\prec X$ and $\GH(Y, Z)\le 2\ep$.
\end{lem}

\begin{proof}
By assumption, there exists a map $f:X\to Y$ with \eqref{ineq;f} for any $x, x' \in X$ and $B_\ep (f(X))=Y$.
Let $\iota:Y\to\R^\infty$ be an isometric embedding.
Define functions $f_n, g_n :X\to\R$ by
\[
(f_n (x))_{n=1}^\infty = \iota\circ f(x) \quad\text{ and }\quad
g_n (x):= \inf_X [ f_n (\cdot)+d_X (x, \cdot) ].
\]

The map $g=(g_n)_{n=1}^\infty :X\to\R^\infty$ is $1$-Lipschitz and satisfies $\|\iota\circ f(x) -g(x)\|_\infty\le\ep$ for any $x\in X$.
Then $Z:=g(X)\subset\R^\infty$ is a compact metric space with $Z\prec X$ and
\[
\GH(Y, Z) \le \haus(\iota(Y), \iota\circ f(X)) + \haus(\iota\circ f(X), g(X)) \le 2\ep,
\]
where $\haus$ is the Hausdorff distance between compact sets in $\R^\infty$.
\end{proof}

\begin{lem}\label{lem;surj}
Let $\mu$ and $\nu$ be Borel measures on separable metric spaces $X$ and $Y$ respectively.
Suppose that a continuous map $f:X\to Y$ satisfies $f_* \mu=\nu$ and that $\supp\mu$ is compact.
Then $f(\supp\mu)=\supp\nu$.
\end{lem}

\begin{proof}
If $y\in Y\setminus f(\supp\mu)$,
then there exists an open set $O\subset Y\setminus f(\supp\mu)$ with $y\in O$.
Then $f^{-1}(O)\subset X\setminus\supp\mu$ and hence
\[
\nu(O)=\mu(f^{-1}(O))=0.
\]
This means that $y\notin\supp\nu$ and hence $\supp\nu\subset f(\supp\mu)$.

For any $y=f(x)\in f(\supp\mu)$ with $x\in\supp\mu$ and any open set $O\subset Y$ with $y\in O$,
we have
\[
\nu(O)=\mu(f^{-1}(O))>0,
\]
which implies that $y\in\supp\nu$ and hence $f(\supp\mu)\subset\supp\nu$.
\end{proof}

\begin{lem}\label{lem;cptdominate}
If $\cP$ is a pyramid and $Y, Z\in\cP$ are compact mm-spaces,
then there exists a compact mm-space $X\in\cP$ with $Y, Z\prec X$.
\end{lem}

\begin{proof}(cf.~Proof of Lemma~\ref{lem;dominatedsubconv})
By definition,
there exists an mm-space $W\in\cP$ with $Y, Z\prec W$.
Let $f_V :W\to V$ be a $1$-Lipschitz map with $(f_V)_* \muw =\mu_V$ for $V=Y, Z$.
For $w, w' \in W$, define
\[
d'_W (w, w') := \max_{V=Y, Z} d_V (f_V (w), f_V (w')),
\]
which is a pseudo-metric on $W$.
Let $X=(X, d_X)$ be the metric completion of the quotient metric space of $(W, d'_W)$ by identifying two points $w, w'$ if $d'_W (w,w')=0$. Let $\pi:W\to X$ be the canonical $1$-Lipschitz map and $\mux:=\pi_* \muw$, so that $X\in\cP$. We define maps $g_V :X\to V$ by
\[
g_V|_{\pi(W)}([w]) := f_V (w)
\]
for $[w] \in \pi(W)$ and $V=Y, Z$.
Here ${g_V}|_{\pi(W)}$ is well-defined, 1-Lipschitz, and hence there exists its unique 1-Lipschitz extension $g_V$ on $X$.
Then it satisfies $(g_V)_* \mux=(g_V \circ \pi)_* \muw =(f_V)_* \muw =\mu_V$.
In particular, we have $Y, Z\prec X$.

Given a sequence $\seqn{x_n}$ of $X$,
there exists a sequence $\seqk{n(k)}$ for which $\seqk{g_V (x_{n(k)})}$ converges in $V=Y, Z$.
Since
\[
d_X (x_{n(k)}, x_{n(l)}) = \max_{V=Y, Z} d_V (g_V (x_{n(k)}), g_V (x_{n(l)})),
\]
$\seqk{x_{n(k)}}$ also converges and hence $X$ is compact.
\end{proof}

\begin{lem}\label{lem;X}
Let $\seqn{X_n}$ be a sequence of compact metric spaces
and $\seqn{\ep_n}$ a sequence of positive numbers with $\sum_{n=1}^\infty \ep_n <\infty$.
Suppose that there exist maps $f_n :X_{n+1} \to X_n$ with
\[
d_{X_{n+1}} (x, x') -\ep_n \le d_{X_n} (f_n (x), f_n (x'))
\]
for any $n\in\N$ and $x, x' \in X_{n+1}$.
Then there exists a compact metric space $X$ with $\GH(X_n, X) \to 0$ as $\nti$.
\end{lem}

\begin{proof}
Take finite subsets $N_n \subset X_n$ with $X_n = U_{\ep_n} (N_n)$.
Let $\pi_n : X_n \to N_n$ be a map with $d_{X_n} (x, \pi_n (x))=d_{X_n} (x, N_n)$ for any $x\in X_n$.
Then the map $g_n :=\pi_n \circ f_n :X_{n+1} \to N_n$ satisfies
\[
d_{X_{n+1}} (x, x') -3\ep_n < d_{X_n} (g_n (x), g_n (x'))
\]
for any $n\in\N$ and $x, x' \in X_{n+1}$.

Define
\[
X':=\left\{(x_n)\in\prod_{n=1}^\infty N_n : g_n (x_{n+1})=x_n \text{ for any } n \right\}
\]
and
\[
d_{X'} (x, x'):= \lim_\nti d_{X_n} (x_n, x'_n) \in [0, \infty)
\]
for $x=(x_n), x'=(x'_n) \in X'$.

Define $X$ as the quotient set of $X'$ with the equivalence relation $\sim$ defined by
$x\sim x'$ if and only if $d_{X'} (x, x')=0$ for $x, x' \in X'$.
Then $(X, d_X)$ equipped with the distance $d_X$ induced by $d_{X'}$ is a compact metric space.

Let us prove that $X$ is compact.
Let $\seqk{x_k}$ be a sequence in $X$.
Take $x_{n, k}\in N_n$ with $(x_{n, k})\in X'$ and $[(x_{n, k})]=x_k$ for each $k\in\N$.
Since $N_n$ is a finite set,
there exist $y_n \in N_n$ and infinite subsets $K_n \subset\N$
with $x_{n, k} =y_n$ and $K_{n+1} \subset K_n \setminus\{\min K_n\}$ for any $n\in\N$ and $k\in K_n$.
Then
\begin{align}\label{eq;y}
g_n (y_{n+1}) =g_n (x_{n+1, k}) =x_{n, k}=y_n
\end{align}
for any $n\in\N$ and $k\in K_{n+1}$ and hence $(y_n)\in X'$. Put $y:=[(y_n)]\in X$.
Define $k(l):=\min K_l$ for $l\in\N$.
For $\ep>0$, take $L\in\N$ with $3\sum_{n=L}^\infty \ep_n <\ep$.
If $l>L$, then we have
\[
d_X (x_{k(l)}, y) =\lim_\nti d_{X_n} (x_{n, k(l)}, y_n) < d_{X_L} (x_{L, k(l)}, y_L) + \ep =\ep
\]
and hence $\seql{x_{k(l)}}$ converges to $y$. Thus $X$ is compact.

Let $\ep>0$ and take a finite subset $N\subset X$ with $X=U_\ep (N)$. Note that $\GH(N, X) < \ep$.
For each $x\in X$, fix $(x_n) \in X'$ with $[(x_n)]=x$.
Define the map $h_n : X\to X_n$ by $h_n (x)=x_n$.

We shall prove that
\[
\limsup_{\nti} \max_{x\in X_n} d_{X_n} (x, h_n (N)) < \ep,
\]
which implies $\limsup_{\nti}\GH(X_n, N) < 2\ep$ by e.g.~\cite[Theorem~6.14]{T}.
If not, there exist an infinite subset $K=K_0 \subset\N$ and $x_k \in X_k$ for $k\in K$ with
\[
\lim_{K\ni k\to\infty} d_{X_k} (x_k, h_k (N)) \ge \ep.
\]

For $(n, k)\in\N\times K$ with $n<k$, we define
\[
x_{n, k} := g_n \circ\dots\circ g_{k-1} (x_k) \in N_n.
\]
Since $N_n$ is a finite set,
by the same argument as above,
there exist $y_n \in N_n$ and infinite subsets $K_n \subset\N$
with $K_n \subset K_{n-1}$ and $x_{n, k} = y_n$ for any $n\in\N$ and $k\in K_n$.
Then Equation~\eqref{eq;y} holds for any $n\in\N$ and $k\in K_{n+1}$ and hence $y:=(y_n)\in X'$.
Take $z=(z_n)\in X'$ with $[z]\in N$ and $d_{X'} (y, z)<\ep$.
Then we have
\begin{align*}
d_{X_k} (x_k, z_k) -d_{X_k} (y_k, z_k)
\le d_{X_k} (x_k, y_k)
\le d_{X_n}(x_{n, k}, y_{n}) + 3\sum_{i=n}^{k-1} \ep_i
\end{align*}
for any $n\in\N$ and $k\in K_n$ with $n<k$,
which yields $\ep-d_{X'} (y, z)\le 0$ as $k\to\infty$ and $\nti$, but this is a contradiction.

Since $\ep>0$ is taken arbitrarily, we conclude that $\GH(X_n, X)\to 0$ as $\nti$.
\end{proof}

\begin{cor}\label{cor;X}
Let $\seqn{X_n}$ be a sequence of compact metric spaces
and $\seqn{\ep_n}$ a sequence of positive numbers with $\sum_{n=1}^\infty \ep_n <\infty$.
Suppose that $X_{n+1} \prec_{\ep_n} X_{n}$ for any $n$.
Then there exists a compact metric space $X$ with $\GH(X_n, X) \to 0$ as $\nti$.
\end{cor}

\begin{proof}
By assumption, there exist maps $g_n :X_n \to X_{n+1}$ with
\[
d_{X_{n+1}} (g_n (x), g_n (x')) \le d_{X_n} (x, x') +\ep_n
\]
for any $x, x' \in X_n$ and $\max_{X_{n+1}} d_{X_{n+1}} (\cdot, g_n (X_n)) \le\ep_n$.

Take maps $\pi_n :X_{n+1} \to g_n(X_n)$ with $d_{X_{n+1}} (x, \pi_n (x))<2\ep_n$ for any $x\in X_{n+1}$
and maps $h_n :X_{n+1} \to X_n$ with $g_n (h_n (x))=x$ for any $x\in g_n (X_n)$.
Define $f_n := h_n \circ\pi_n :X_{n+1}\to X_n$.
Then
\begin{align*}
d_{X_n} (f_n (x), f_n (x')) +\ep_n
&\ge d_{X_{n+1}} (g_n (f_n (x)), g_n (f_n (x'))) \\
&= d_{X_{n+1}} (\pi_n (x), \pi_n (x'))
> d_{X_{n+1}} (x, x') -4\ep_n
\end{align*}
for any $x, x' \in X_{n+1}$.
Now Corollary~\ref{cor;X} follows from Lemma~\ref{lem;X}.
\end{proof}

\section{Proof of Theorem~\ref{thm;main}}\label{sec;pf}
In this section, we prove Theorem~\ref{thm;main} and Proposition~\ref{prop;main}.

\begin{proof}[Proof of Theorem~\ref{thm;main}]
Let $\cP\subset\cX$ be a pyramid,
$\cY\subset\cP$ a GH-precompact subset,
and $\seqn{\ep_n}$ a sequence of positive numbers with $\sum_{n=1}^\infty \ep_n <\infty$.

We choose a countable dense subset $\cD=\{Y_1, Y_2, \dots\}\subset\cY$
and apply Lemma~\ref{lem;cptdominate} to take a sequence $\seql{W_l}$ of compact mm-spaces $W_l \in\cP$ with $W_1 =Y_1$ and $Y_l, W_{l-1} \prec W_l$ for any $l\ge 2$.
Let $f_l :W_{l+1} \to W_l$ and $g_l :W_l \to Y_l$ be $1$-Lipschitz maps with $(f_l)_* \mu_{W_{l+1}} =\mu_{W_l}$ and $(g_l)_* \mu_{W_l} =\mu_{Y_l}$ respectively.
By Lemma~\ref{lem;surj}, these maps are surjective.

We use the GH-precompactness of $\cY$ and Lemma~\ref{lem;GHprecpt} to take a finite set
\[
\cA_{Y, n} =\{A_{Y, n, 1}, \dots, A_{Y, n, N(Y, n)}\}
\]
of disjoint Borel sets $A_{Y, n, i} \subset Y$ with
\begin{itemize}
\item $\diam A<\ep_n$ if $Y\in\cY$ and $A\in\cA_{Y, n}$,
\item $\bigsqcup\cA_{Y, n}=\bigsqcup_{i=1}^{N(Y, n)} A_{Y, n, i} =Y$ for any $Y\in\cY$,
\item $\sup_{Y\in\cY} \#\cA_{Y, n} = \sup_{Y\in\cY} N(Y, n) <\infty$, and
\item any $B\in\cA_{Y, n+1}$ satisfies $B\subset A$ for some $A\in\cA_{Y, n}$.
\end{itemize}

Choose $\pt_{Y, n, i} \in A$ for each $A=A_{Y, n, i} \in\bigcup_{n\in\N} \bigcup_{Y\in\cD} \cA_{Y, n}$.

Letting $D:=\sup_{Y\in\cY} \diam Y <\infty$,
we use the $\Box$-precompactness of $\cY$ and the compactness of the closed interval $[0, D]\subset\R$
to take a sequence $\seqn{\cY_n}$ of finite subsets $\cY_n \subset\cD$ and $\cY_{Y, n}\subset\cY_n$ with
\begin{itemize}
\item $\Box(Y, \cY_n)<\ep_n$ for any $Y\in\cD$,
\item $\cY_n \subset\cY_{n+1} =\bigcup_{Y\in\cY_n} \cY_{Y, n+1}$, and
\item $Y\in\cY_{Y, n+1}$, $N(Y, n)=N(Z, n)$, and
\begin{equation}\label{ineq;pt}
|d_Y (\pt_{Y, n, i}, \pt_{Y, n, j})-d_Z (\pt_{Z, n, i}, \pt_{Z, n, j})| <\ep_n,
\end{equation}
if $Y\in\cY_n$, $Z\in\cY_{Y, n+1}$, and $1\le i, j\le N(Y, n)$.
\end{itemize}

Such $\cY_n$ and $\cY_{Y, n}$ are obtained inductively as follows.
First, we use the $\Box$-precompactness of $\cD$ to take a finite subset $\cD_n \subset\cD$ with $\Box(Y, \cD_n)<\ep_n$ for any $Y\in\cD$.
Next, we use $\sup_{Y\in\cD} N(Y, n) <\infty$ and the compactness of $[0, D]$ to take a finite subset $\cY_n \subset\cD$ and subsets $\cD_{Y, n+1} \subset\cD$ for $Y\in\cY_n$ with
\begin{itemize}
\item $\cD_n \cup\cY_{n-1} \subset\cY_n$, where $\cY_0 :=\emptyset$,
\item $\cD=\bigcup_{Y\in\cY_n} \cD_{Y, n+1}$, and
\item $Y\in\cD_{Y, n+1}$, $N(Y, n)=N(Z, n)$, and Inequality~\eqref{ineq;pt} holds if $Y\in\cY_n$, $Z\in\cD_{Y, n+1}$, and $1\le i, j\le N(Y, n)$.
\end{itemize}
Then we define
$\cY_{Y, n+1} :=\cY_{n+1} \cap\cD_{Y, n+1}$ for $Y\in\cY_n$.

We fix $n\in\N$ and put $L(n):=\max\{l:Y_l \in\cY_n\}$.
Then $Y\prec W_{L(n)}$ for any $Y\in\cY_n$.
Let $h_{n, Y} := g_l \circ f_l \circ\dots\circ f_{L(n)-1} :W_{L(n)} \to Y$ for $Y=Y_l \in\cY_n$.
Define $X_n$ (resp. $\overline{X}_n$) as the set of all functions $\phi:\cY_n \to\N$
with $1\le \phi(Y)\le N(Y, n)$ for any $Y\in\cY_n$ and $\mu_{W_{L(n)}} (A_\phi)>0$ (resp. $A_\phi \ne\emptyset$), where
\[
A_\phi :=\bigcap_{Y\in\cY_n} h_{n, Y}^{-1} (A_{Y, n, \phi(Y)}) \subset W_{L(n)},
\]
and
\begin{align*}
d_{X_n}(\phi, \psi) &:= \max\{ d_Y (\pt_{Y, n, \phi(Y)}, \pt_{Y, n, \psi(Y)}) : Y\in\cY_n \}, \\
\mu_{X_n} (\{\phi\}) &:= \mu_{W_{L(n)}} (A_\phi)
\end{align*}
for $\phi, \psi\in X_n$. Note that $\mu_{X_n}$ is a probability measure on $X_n$.
Thus $X_n =(X_n, d_{X_n}, \mu_{X_n})$ is a finite mm-space.

Define the set $V_n$ of set-valued maps $\phi:\cY_n \to 2^\N \setminus\{\emptyset\}$ with $\phi(Y)\subset\{1, \dots, N(Y, n)\}$ for any $Y\in\cY_n$.
Then define
\[
d_{V_n}(\phi, \psi):=\max\{ d_Y (\pt_{Y, n, i}, \pt_{Y, n, j}) : Y\in\cY_n, i\in\phi(Y), j\in\psi(Y) \}
\]
if $\phi\ne\psi$ and $d_{V_n}(\phi, \psi):=0$ if $\phi=\psi$ for $\phi, \psi\in V_n$
to obtain a finite metric space $V_n =(V_n, d_{V_n})$.

\begin{clm}\label{cl;V}
There exists a compact metric space $V$ with $\GH(V_n, V)\to 0$ as $\nti$.
\end{clm}

\begin{proof}
Let $\phi\in V_{n+1}$ and suppose that $Y\in\cY_n$, $Z\in\cY_{Y, n+1} \subset\cY_{n+1}$, and $i\in\phi(Z)$.
Then there exists $i'\in\N$ with $1\le i'\le N(Y, n)=N(Z, n)$ and
\begin{equation}\label{ineq;A}
A_{Z, n+1, i} \subset A_{Z, n, i'}.
\end{equation}

Take $\phi':\cY_n \to 2^\N$ with $\phi'(Y)=\{i' : Z\in\cY_{Y, n+1}, i\in\phi(Z)\}$ for any $Y\in\cY_n$.
Then $\phi' \in V_n$.
This defines a map $V_{n+1} \ni\phi \mapsto \phi' \in V_n$.

If $\phi, \psi\in V_{n+1}$, $Y\in\cY_n$, $Z\in\cY_{Y, n+1} \subset\cY_{n+1}$, $i\in\phi(Z)$, and $j\in\psi(Z)$ satisfy
\[
d_{V_{n+1}}(\phi, \psi) = d_Z (\pt_{Z, n+1, i}, \pt_{Z, n+1, j}),
\]
then
\begin{align*}
d_{V_{n+1}}(\phi, \psi)
&< d_Z (\pt_{Z, n, i'}, \pt_{Z, n, j'}) + 2\ep_n \\
&<d_Y (\pt_{Y, n, i'}, \pt_{Y, n, j'}) +3\ep_n
\le d_{V_n}(\phi', \psi') + 3\ep_n,
\end{align*}
where the first and second inequalities follow from Inequalities~\eqref{ineq;A} and \eqref{ineq;pt} respectively.

Now Claim~\ref{cl;V} follows from Lemma~\ref{lem;X}.
\end{proof}

Any element of $X_n$ can be regarded as an element of $V_n$ and this gives an isometric embedding $X_n \hookrightarrow V_n$.
By Claim~\ref{cl;V} and Lemmas~\ref{lem;precpt} and \ref{lem;GHprecpt},
a subsequence of $\seqn{X_n}$ $\Box$-converges to some mm-space $X$.

\begin{clm}
$X\in\cP$.
\end{clm}

\begin{proof}
We fix $n\in\N$ and define a Borel map $\Phi_n: W_{L(n)} \to\overline{X}_n$ by
$\Phi_n (w) :=\phi$ if $w \in A_\phi$ and $\phi\in\overline{X}_n$.
Then we have $(\Phi_n)_* \mu_{W_{L(n)}} = \mu_{X_n}$. Moreover, for any $w \in A_\phi$ and $w' \in A_\psi$, we have
\begin{align*}
d_{X_n}(\phi, \psi)
&= \max_{Y \in \cY_n} d_Y(\pt_{Y, n, \phi(Y)}, \pt_{Y,n,\psi(Y)})\\
&< \max_{Y \in \cY_n} d_Y(h_{n, Y}(w), h_{n, Y}(w')) + 2\ep_n
\le d_{W_{L(n)}}(w, w') + 2\ep_n,
\end{align*}
where the first inequality follows from
\[
d_Y (\pt_{Y, n, \phi(Y)}, h_{n, Y}(w)) \le \diam{A_{Y,n,\phi(Y)}} <\ep_n
\]
and the similar inequality for $\psi$ and $w'$.

Thus $X_n \prec_{2\ep_n} W_{L(n)}$. Since $W_{L(n)} \in\cP$, we obtain $X\in\cP$ as $\nti$ by Lemma~\ref{lem;K}.
\end{proof}

\begin{clm}
$Y\prec X$ for any $Y\in\cY$.
\end{clm}

\begin{proof}
Fix $Y\in\cY_m$ for some $m\in\N$.
Define a map $\Psi_n :X_n \to Y$ for $n \ge m$ by $\Psi_n (\phi) =\pt_{Y, n, \phi(Y)}$.
Then $\Psi_n$ is $1$-Lipschitz. Moreover, we have
\begin{align*}
(\Psi_n)_* \mu_{X_n}
&= \sum_{\phi \in X_n} \mu_{W_{L(n)}}(A_\phi) \, \delta_{\pt_{Y,n,\phi(Y)}}\\
&= \sum_{i=1}^{N(Y, n)} \left(\sum_{\phi \in X_n; \, \phi(Y)=i} \mu_{W_{L(n)}}(A_\phi)\right) \delta_{\pt_{Y,n,i}}
= \sum_{i=1}^{N(Y, n)} \muy(A_{Y, n, i}) \, \delta_{\pt_{Y, n, i}}
\end{align*}
and hence
\[
(\Psi_n)_* \mu_{X_n}(A) \le \muy(U_{\ep_n}(A))
\]
for any Borel set $A \subset Y$, which leads to $\proh((\Psi_n)_* \mu_{X_n}, \muy) \le \ep_n$.
Then we obtain $Y\prec_{\ep_n} X_n$ and then $Y\prec X$ as $\nti$ by Lemma~\ref{lem;precep}.

Since $\bigcup_\niN \cY_n$ is dense in $\cY$,
we also have $\cY\prec X$ by Lemma~\ref{lem;precep} again.
\end{proof}

Now our proof of Theorem~\ref{thm;main} is complete.
\end{proof}

\begin{rem}
In fact, the sequence $\seqn{X_n}$ itself $\Box$-converges to $X$ in our proof of Theorem~\ref{thm;main}.

The map $f_n :X_{n+1}\to X_n; \phi\mapsto\phi''$,
where $\phi'' \in X_n$ is the element with $A_{Y, n+1, \phi(Y)} \subset A_{Y, n, \phi''(Y)}$ for any $Y\in\cY_n$,
satisfies
\begin{align*}
d_{X_n} (f_n (\phi), f_n (\psi))
&=d_{X_n} (\phi'', \psi'') \\
&=\max\{ d_Y (\pt_{Y, n, \phi''(Y)}, \pt_{Y, n, \psi''(Y)}) : Y\in\cY_n \} \\
&<\max\{ d_Y (\pt_{Y, n+1, \phi(Y)}, \pt_{Y, n+1, \psi(Y)}) : Y\in\cY_n \} +2\ep_n \\
&\le\max\{ d_Y (\pt_{Y, n+1, \phi(Y)}, \pt_{Y, n+1, \psi(Y)}) : Y\in\cY_{n+1} \} +2\ep_n \\
&=d_{X_{n+1}} (\phi, \psi) +2\ep_n
\end{align*}
for any $\phi, \psi\in X_{n+1}$ and $(f_n)_* \mu_{X_{n+1}} = \mu_{X_n}$.
Thus $X_n \prec_{2\ep_n} X_{n+1}$ for any $n$.

Since $\seqn{X_n}$ is $\Box$-precompact by Claim~\ref{cl;V} and Lemmas~\ref{lem;precpt} and \ref{lem;GHprecpt},
$\seqn{X_n}$ $\Box$-converges to $X$ by Proposition~\ref{prop;epincrease}.
\end{rem}

\begin{proof}[Proof of Proposition~\ref{prop;main}]
The proof is the same as that of Theorem~\ref{thm;main}
except for the use of Lemmas~\ref{lem;GHprecep} and \ref{lem;KGH} instead of Lemmas~\ref{lem;precep} and \ref{lem;K}.
\end{proof}

\section{Miscellaneous results}\label{sec;mis}
In this section, we collect several results related to our Theorem~\ref{thm;main}.

First, we construct examples of $\cY$ and $X$ as in Theorem~\ref{thm;G}.
\begin{ex}[cf.~{\cite[Theorem~4.27]{S}}]
\label{ex;XDN}
Let $\cX^D_{\le N}$ be the set of all (mm-isomorphism classes of) finite mm-spaces with cardinality $\le N$ and diameter $\le D$
for $N\in\N\cap [2, \infty)$ and $D>0$.
Define a sequence $\seqn{p_n}$ by
\[
p_1 :=\frac{1}{N} \quad\text{ and }\quad p_{n+1}:=\frac{1}{N}\left(1-\sum_{k=1}^{n} p_k \right)
\]
for $n\in\N$.
For example, $p_n =2^{-n}$ if $N=2$.

Define $X=X_N :=\N$ and
\[
d_X (x, x') := \begin{cases}
\ D &\text{ if } x\ne x' \\
\ 0 &\text{ if } x=x'
\end{cases}
\qquad\text{ and }\qquad
\mux(\{x_n\}):=p_n
\]
for $x_n \in X$.
Then $\cX^D_{\le N} \prec X$.

Indeed, if $Y=\{y_1, \dots, y_M\}\in\cX^D_{\le N}$, then
there exist $a_{k, n} \in\{0, 1\}$ with
\[
\muy(\{y_k\})=\sum_{n=1}^\infty a_{k, n} p_n \qquad\text{ and }\qquad \sum_{k=1}^N a_{k, n}=1.
\]
Define $f:X\to Y$ by $f(x_n)=y_k$ if $a_{k, n}=1$.
Then $f_* \mux=\muy$ and hence $Y\prec X$.
\end{ex}

The following is an application of Theorem~\ref{thm;main}.
\begin{prop}[cf.~Lemma~\ref{lem;dominatedsubconv}]
Let $\seqn{\ep_n}$ be a sequence of positive numbers with $\ep_n \to 0$ as $\nti$.
Suppose that $\seqn{X_n}$ is a sequence of mm-spaces,
$\cY$ is a set of mm-spaces, $\{\supp\muy : Y\in\cY\}$ is GH-precompact,
and $Y\prec_{\ep_n} X_n$ for any $n$ and $Y\in\cY$.
Then there exists a sequence $\seqn{Z_n}$ of mm-spaces with $Z_n \prec X_n$ for any $n$,
whose subsequence converges to an mm-space $X$ with $\cY\prec X$.
\end{prop}

\begin{proof}
By using Theorem~\ref{thm;rho} and taking a subsequence,
we may assume that $\seqn{\cP_{X_n}}$ converges weakly to some pyramid $\cP$.
Then $\cY\subset\cP$ by Lemma~\ref{lem;K}.
Now Theorem~\ref{thm;main} states that there exists an mm-space $X$ with $\cY\prec X\in\cP$.
Then there exists a sequence $\seqn{Z_n}$ of $Z_n \in\cP_{X_n}$ which $\Box$-converges to $X$.
This finishes the proof.
\end{proof}

\begin{lem}\label{lem;precinjective}
Let $\cY\subset\cX$ be a nonempty subset which is downward directed,
i.e., if $X, Y\in\cY$, then there exists $Z\in\cY$ with $Z\prec X, Y$.
Then $\cP:=\bigcap_{Y\in\cY} \cP_Y$ is a pyramid and there exists an mm-space $X$ with $\cP=\cP_X$.
\end{lem}

\begin{proof}
Since $\cP\subset\cX$ contains the one-point mm-space,
$\cP$ is nonempty and closed.
If $Z\in\cP$ and $W\prec Z$, then $W\in\cP_Y$ for any $Y\in\cY$ and hence $W\in\cP$.

Let  $Z, W\in\cP$ and
$\seqn{\ep_n}$ a sequence of positive numbers with $\ep_n \to 0$ as $\nti$.
Fix $Y_0 \in\cY$ and put $\cP_0 :=\cP_{Y_0}$ and $\cY_0 :=\cP_0 \cap\cY$.
Since $\cP_0$ is compact by Lemma~\ref{lem;precpt},
there exists a finite subset $\cY_n \subset\cY_0$ with $\Box(Y, \cY_n)<\ep_n$ for any $Y\in\cY_0$.
By assumption, there exists $V_n \in\cY$ with $V_n \prec Y$ for any $Y\in\cY_n$.
Since $Z, W\in\cP\subset\cP_{V_n}$,
a subsequence of a sequence $\seqn{X_n}$ of mm-spaces with $Z, W\prec X_n \prec V_n$ $\Box$-converges to some mm-space $X$ with $Z, W\prec X$ by Lemmas~\ref{lem;dominatedsubconv} and \ref{lem;precep}.

If $Y\in\cY$, there exists $Y' \in\cY_0$ with $Y' \prec Y$
and take a sequence $\seqn{Y_n}$ with $Y_n \in\cY_n$ and $\Box(Y_n, Y')<\ep_n$ and hence $Y_n \prec_{3\ep_n} Y$ by Remark~\ref{rem;prec}.
Then we have
\[
X_n \in\cP_{V_n} \subset \cP_{Y_n} \subset U_{9\ep_n} (\cP_Y)
\]
for any $n$ by Lemma~\ref{lem;K} and hence $X\in\cP_Y$.
Thus $X\in\cP$ and hence $\cP$ is a pyramid.

Since $\cP\subset\cP_0$,
$\cP$ is a compact pyramid and hence
there exists an mm-space $X$ with $\cP=\cP_X$ by e.g. \cite[Lemma~7.14]{S} and Lemma~\ref{lem;precep}.
\end{proof}

We recall that $\Pi$ denotes the set of all pyramids.
The following is a generalization of Lemma~\ref{lem;precinjective}.
\begin{lem}\label{lem;pyramidinjective}
Let $\Sigma\subset\Pi$ be a nonempty set of pyramids which is downward directed,
i.e., if $\cP, \cQ\in\Sigma$, then there exists $\mathcal{R}\in\Sigma$ with $\mathcal{R}\subset\cP\cap\cQ$.
Then $\cP:=\bigcap\Sigma$ is a pyramid.
\end{lem}

\begin{proof}
Since $\cP\subset\cX$ contains the one-point mm-space,
$\cP$ is nonempty and closed.
If $X\in\cP$ and $Y\prec X$, then $Y\in\cQ$ for any $\cQ\in\Sigma$ and hence $Y\in\cP$.

Let $Y, Z\in\cP$ and $\seqn{\ep_n}$ a sequence of positive numbers with $\ep_n \to 0$ as $\nti$.
Since $(\Pi, \rho)$ is a compact metric space,
where $\rho$ is the metric given in Theorem~\ref{thm;rho},
there exists finite subsets $\Sigma_n \subset\Sigma$ with $\rho(\cQ, \Sigma_n)<\ep_n$ for any $\cQ\in\Sigma$.
By assumption, there exists a pyramid $\cP_n \in\Sigma$ with $\cP_n \subset\bigcap\Sigma_n$.
Since $Y, Z\in\cP\subset\cP_n$,
there exists $W_n \in\cP_n$ with $Y, Z\prec W_n$ and a subsequence of $\seqn{X_n}$ with $Y, Z\prec X_n \prec W_n$ $\Box$-converges to some mm-space $X$ with $Y, Z\prec X$ by Lemmas~\ref{lem;dominatedsubconv} and \ref{lem;precep}.

If $\cQ\in\Sigma$,
take a sequence $\seqn{\cQ_n}$ with $\cQ_n \in\Sigma_n$ and $\rho(\cQ_n, \cQ)<\ep_n$.
Then we have
\[
X_n \in\cP_{W_n} \subset\cP_n \subset\bigcap\Sigma_n \subset\cQ_n
\]
for any $n$ and hence $X\in\cQ$.
Thus $X\in\cP$ and hence $\cP$ is a pyramid.
\end{proof}

Lemmas~\ref{lem;precinjective} and \ref{lem;pyramidinjective} are not trivial because an intersection of pyramids is not necessarily a pyramid in general.
\begin{ex}
There are pyramids $\cP$ and $\cQ$ for which $\cP\cap\cQ$ is not a pyramid.

Let $(A=\{a, b, c, d\}, d_A)$ be a metric space with the distances between two distinct points being $1$ and define
\begin{align*}
X&:=(A, d_A, (1/2)\delta_a +(1/3)\delta_b +(1/6)\delta_c), \\
Y&:=(A, d_A, (5/12)\delta_a +(5/12)\delta_b +(1/12)\delta_c +(1/12)\delta_d), \\
Z&:=(A, d_A, (1/2)\delta_a +(1/2)\delta_b), \text{ and } \\
W&:=(A, d_A, (5/6)\delta_a +(1/6)\delta_c).
\end{align*}

Then $Z, W\in\cP_X \cap\cP_Y$, but there does not exist $V\in\cP_X \cap\cP_Y$ with $Z, W\prec V$.
Thus $\cP_X \cap\cP_Y$ is not a pyramid.
\end{ex}

The following gives a minimal mm-space $Z$ with $\cY\prec Z$ for a subset $\cY\subset\cX$.
\begin{prop}\label{prop;minimalX}
Let $\cY\subset\cX$ be a subset.
Suppose that there exists an mm-space $X$ with $\cY\prec X$.
Then there exists an mm-space $Z$ such that
\begin{enumerate}
\item $\cY\prec Z\prec X$ and
\item any mm-space $W$ with $\cY\prec W\prec Z$ is mm-isomorphic to $Z$.
\end{enumerate}
\end{prop}

\begin{proof}
Let $\cZ:=\{Z\in\cX : \cY\prec Z\prec X\} \ne\emptyset$.
Then Lemma~\ref{lem;precinjective} implies that $(\cZ, \prec)$ is an inductive ordered set.
Thus Zorn's lemma implies that there exists $Z\in\cZ$ for which any $W\in\cZ$ with $W\prec Z$ is mm-isomorphic to $Z$.
\end{proof}

The following gives a minimal pyramid $\cQ$ with $\cY\subset\cQ$ for a subset $\cY\subset\cX$.
\begin{prop}\label{prop;minimalP}
Let $\cY\subset\cX$ be a subset.
Suppose that there exists a pyramid $\cP$ with $\cY\subset\cP$.
Then there exists a pyramid $\cQ$ such that
\begin{enumerate}
\item $\cY\subset\cQ\subset\cP$ and
\item any pyramid $\mathcal{R}$ with $\cY\subset\mathcal{R}\subset\cQ$ satisfies $\mathcal{R}=\cQ$.
\end{enumerate}
\end{prop}

\begin{proof}
The proof is the same as that of Proposition~\ref{prop;minimalX} except for the use of Lemma~\ref{lem;pyramidinjective} instead of Lemma~\ref{lem;precinjective}.
\end{proof}

\begin{lem}\label{lem;order}
Let $N\in\N$.
Suppose that sequences $\{X_n\}_{n=1}^N$ of mm-spaces and
$\{\ep_n\}_{n=2}^N$ of nonnegative numbers satisfy $X_{n-1} \prec_{\ep_n} X_{n}$ for any $n\ge 2$.
Put $\ep:=\sum_{n=2}^{N} \ep_n \ge 0$.
Then $X_1 \prec_{5\ep} X_N$.
\end{lem}

\begin{proof}
By assumption,
there exist Borel maps $f_n :X_{n} \to X_{n-1}$ and Borel sets $\tX_{n}\subset X_{n}$ with
\begin{itemize}
\item $\mu_{X_{n}} (\tX_{n}) \ge 1-\ep_n$,
\item $d_{X_{n-1}} (f_n (x), f_n (x'))\le d_{X_{n}} (x, x)+\ep_n$ for any $x, x'\in\tX_{n}$, and
\item $\proh((f_n)_* \mu_{X_{n}}, \mu_{X_{n-1}})\le\ep_n$.
\end{itemize}

We may assume that $X_n$'s are isometrically embedded in $(\R^\infty, \|\cdot\|_\infty)$ by the Kuratowski embedding and
define functions $f_{n, i} :X_n \to\R$ and $g_{n, i} :\R^\infty \to\R$ by
\[
(f_{n, i} (x))_{i=1}^\infty = f_n (x) \quad\text{ and }\quad
g_{n, i} (x):= \inf_{\tX_n} [ f_{n, i} (\cdot)+d_{\R^\infty} (x, \cdot) ].
\]
Then $|f_{n, i} -g_{n, i}|\le\ep_n$ on $\tX_n$, cf.~\cite[Proof of Lemma~5.4]{S},
and $g_n =(g_{n, i}) :\R^\infty \to\R^\infty$ is $1$-Lipschitz.

Put $X:=X_N$, $Y:=X_1$, and $g:= g_2 \circ\dots\circ g_{N} :\R^\infty \to\R^\infty$.
By Lemma~\ref{lem;P},
\begin{align*}
\proh(\muy, g_* \mux)
&\le \sum_{n=2}^{N} \proh( (g_2 \circ\dots\circ g_{n-1})_* \mu_{X_{n-1}}, (g_2 \circ\dots\circ g_n)_*  \mu_{X_n} ) \\
&\le \sum_{n=2}^{N} \proh( \mu_{X_{n-1}}, (g_n)_* \mu_{X_n} ) \\
&\le \sum_{n=2}^{N} \proh( \mu_{X_{n-1}}, (f_n)_* \mu_{X_n} ) + \proh( (f_n)_* \mu_{X_n}, (g_n)_* \mu_{X_n} )
\le 2\ep.
\end{align*}

Put $\tY:=B_{2\ep}(Y)\subset\R^\infty$ and
take a Borel map $h:\R^\infty \to Y$ with
\[
d_{\R^\infty} (z, h(z)) < d_{\R^\infty} (z, Y)+\frac{\ep}{2}
\]
for any $z\in\R^\infty$ by e.g. \cite[Lemma~3.5]{S}.
This $h$ satisfies
\[
d_{\R^\infty} (h(z), z) <\frac{5}{2}\ep \quad\text{ and hence }\quad d_Y (h(z), h(z')) <d_{\R^\infty} (z, z') +5\ep
\]
for any $z, z' \in\tY$.

Put $f:=h\circ g:X\to Y$ and $\tX:=g^{-1}(\tY) \subset X$.
Then
\[
\mux(\tX)=g_* \mux (\tY) \ge \muy(Y)-2\ep =1-2\ep.
\]
If $x, x'\in\tX$, then $g(x), g(x')\in\tY$ and hence
\[
d_Y (f(x), f(x')) <d_{\R^\infty} (g(x), g(x')) +5\ep \le d_X (x, x') +5\ep.
\]
Finally, Lemma~\ref{lem;P} yields
\[
\proh(f_* \mux, \muy)
\le \proh(h_* g_* \mux, g_* \mux) +\proh(g_* \mux, \muy)
< 5\ep.
\]

Therefore we conclude that $Y\prec_{5\ep} X$.
\end{proof}

If $N=2$ in Lemma~\ref{lem;order},
we have the following with slightly smaller constants.
\begin{lem}
Let $X$, $Y$, and $Z$ be three mm-spaces and $\ep, \delta\ge 0$.
If $Z \prec_\delta Y$ and $Y \prec_\ep X$, then $Z \prec_{3\ep+4\delta} X$.
\end{lem}

\begin{proof}
By assumption,
there exist Borel maps $f:X\to Y$ and $g:Y\to Z$ and Borel sets $\tX \subset X$ and $\tY \subset Y$ with $\mux(\tX) \ge 1-\ep$ and $\muy(\tY) \ge 1-\delta$ such that
\[
\sup_{x, x' \in \tX} (d_Y(f(x), f(x'))-d_X(x, x')) \le \ep, \quad \sup_{y, y' \in \tY} (d_Z(g(y), g(y'))-d_Y(y, y')) \le \delta
\]
and
\[
\proh(f_* \mux, \muy) \le \ep, \quad \proh(g_* \muy, \muz) \le \delta.
\]

The set $\hat{X} := \tX \cap f^{-1}(B_\ep(\tY)) \subset X$ has the mass
\[
\mux(\hat{X}) \ge \mux(\tX) +f_* \mux(B_\ep(\tY))-1 \ge \muy(\tY)-2\ep \ge 1-2\ep-\delta.
\]
Define a map $h:X\to Z$ by $h = g\circ \pi \circ f$, where $\pi: Y\to\tY$ is a Borel map with
\[
d_Y (y, \pi(y)) <d_Y (y, \tY)+\delta
\]
for any $y\in Y$, e.g.~\cite[Lemma~3.5]{S}.
Then, for any $x, x' \in \hat{X}$, we have
\begin{align*}
d_Z(h(x),h(x'))
&\le d_Y(\pi\circ f(x), \pi\circ f(x')) + \delta \\
&< d_Y(f(x), f(x')) +2\ep+3\delta\le d_X(x, x') +3\ep+3\delta.
\end{align*}

Moreover, by Lemma~\ref{lem;P} and $f_* \mux(B_\ep (\tY)) \ge \muy(\tY)-\ep \ge 1-(\ep+\delta)$, we have
\begin{align*}
\proh(h_*\mux, \muz)
&\le \proh(h_*\mux, g_*\muy) + \delta \\
&\le \proh(\pi_* f_*\mux, \muy) + 3\delta \\
&\le \proh(\pi_* f_*\mux, f_*\mux) + \ep + 3\delta
\le 2\ep+4\delta.
\end{align*}
These together imply $Z \prec_{3\ep+4\delta} X$. The proof is completed.
\end{proof}

The following is a slight extension of the statements in \cite{G} which deal with the case $\ep_n =0$.
Case~(2) is an analog of Corollary~\ref{cor;X}.
\begin{prop}[cf. Gromov~{\cite[3$\frac{1}{2}$.15]{G}}]\label{prop;epincrease}
Let $\seqn{X_n}$ be a sequence of mm-spaces
and $\seqn{\ep_n}$ a sequence of nonnegative numbers with $\sum_{n=1}^\infty \ep_n <\infty$.
Suppose that either
\begin{enumerate}
\item $X_n \prec_{\ep_n} X_{n+1}$ for any $n$ and $\seqn{X_n}$ is $\Box$-precompact, or
\item $X_{n+1} \prec_{\ep_n} X_n$ for any $n$.
\end{enumerate}
Then there exists an mm-space $Y$ with $\Box(X_n, Y) \to 0$ as $\nti$.
\end{prop}

\begin{proof}
Since the proof for Case (1) is similar, we only discuss Case (2).

Let $\ep>0$, fix $N\in\N$ with $5\sum_{n\ge N} \ep_n <\ep$, and put $X:=X_N$.
In Case (2),
Lemma~\ref{lem;order} implies that $X_n \prec_{\ep} X$ for any $n>N$ and hence
there exist a Borel map $f_n :X\to X_n$ and a Borel set $\tX_n \subset X$ with $\mux(\tX_n)\ge 1-\ep$,
\[
d_{X_n} (f_n (x), f_n (x')) \le d_{X_N} (x, x')+\ep
\]
for any $x, x' \in\tX_n$, and $\proh((f_n)_* \mux, \mu_{X_n})\le\ep$.

Take a finite set $\cK$ of Borel sets of $X$ with
\[
\max_{K\in\cK} \diam K \le\ep \quad\text{ and }\quad
\mux(\bigcup\cK)\ge 1-\ep.
\]

Put $\cK_n :=\{ B_\ep (f_n (K\cap\tX_n)) : K\in\cK \}$
and $U_n :=\bigcup_{K\in\cK} f_n (K\cap\tX_n) \subset X_n$ for $n>N$.
Then $\cK_n$ is a finite set of Borel sets in $X_n$ with
\[
\#\cK_n \le\#\cK, \quad
\max_{K\in\cK_n} \diam K \le 4\ep, \quad
\diam \bigcup\cK_n \le \diam \bigcup\cK+3\ep,
\]
and
\[
\mu_{X_n} (\bigcup\cK_n)
=\mu_{X_n} (B_\ep (U_n))
\ge (f_n)_* \mux(U_n) -\ep
\ge \mux(\bigcup\cK\cap\tX_n) -\ep
\ge  1-3\ep.
\]

Now we deduce from Lemma~\ref{lem;precpt} that $\seqn{X_n}$ is $\Box$-precompact.
Then there exist an increasing sequence $\seqk{n(k)}$ and an mm-space $Y$ with $\Box(X_{n(k)}, Y)\to 0$ as $k\to\infty$.
Lemmas~\ref{lem;order} and \ref{lem;precep} imply that any $\Box$-converging subsequence of $\seqn{X_n}$ $\Box$-converges to $Y$.
This means that $\seqn{X_n}$ $\Box$-converges to $Y$.
\end{proof}

\section{Proofs of Proposition~\ref{prop;embed} and Corollary~\ref{cor;pmG}}\label{sec;embed}
In this section, we prove Proposition~\ref{prop;embed} and Corollary~\ref{cor;pmG}.

\begin{defn}[\cite{GMS}]\label{def;pmG}
By a \emph{pointed mm-space} or a \emph{pmm-space} for short, we mean a pair $(X, x)$ of an mm-space $X$ and a point $x\in\supp\mux$.

We say that a sequence $\seqn{(X_n, x_n)}$ of pmm-spaces \emph{pmG}-converges to a pmm-space $(Y, y)$
if there exist a complete separable metric space $(Z, d_Z)$ and isometric embeddings $f_n :\supp\mu_{X_n} \to Z$ and $f:\supp\muy\to Z$ for which
$\seqn{(f_n)_* \mu_{X_n}}$ converges weakly to $f_* \muy$ and $d_Z (f_n (x_n), f(y))\to 0$ as $\nti$.
\end{defn}

We note that the pmG-convergence makes sense even if the measures are not probability measures.

\begin{lem}\label{lem;diam}
If $\cY\subset\cX$ is a nonempty $\Box$-precompact set, then
\[
\sup_{Y, Z\in\cY} \Box(Y, Z)<1.
\]
\end{lem}

\begin{proof}
Since $\cY\subset\cX$ is nonempty and $\Box$-precompact,
there exist mm-spaces $Y, Z\in\overline{\cY}$ with $\Box(Y, Z)=\sup_{Y, Z\in\cY} \Box(Y, Z)$.
Let $X$ be an mm-space with $Y, Z\prec X$ and
let $f:X\to Y$ and $g:X\to Z$ be $1$-Lipschitz maps with $f_* \mux=\muy$ and $g_* \mux=\muz$ respectively.
Take a Borel set $A \subset X$ such that $\diam{A}<1$ and $\mux(A)>0$.
Then, by Remark \ref{rem;f*X},
\[
\Box(Y, Z) = \Box(f_* X, g_* X) \le \max\{\diam{A}, 1-\mux(A)\} < 1.
\]
The proof is completed.
\end{proof}

\begin{proof}[Proof of Proposition~\ref{prop;embed}]
The implication (2)$\implies$(1) follows from Lemma~\ref{lem;BP}.

To prove the implication (1)$\implies$(2),
we mimic the proof of the Union lemma, e.g.~\cite[Lemma~4.13]{S}.
Let $\cY\subset\cX$ be a $\Box$-precompact subset
and $\seqn{\ep_n}$ a sequence of positive numbers with $\sum_{n=1}^\infty \ep_n < +\infty$ and $\ep_1=1$.
Put $I_1:=\{1\}$ and $\cY_{1,1} := \cY$.
We know that $\diam\cY_{1, 1}<\ep_1$ by Lemma~\ref{lem;diam}. For $n\ge 2$,
we consider a decomposition $\cY=\bigsqcup_{i\in I_n} \cY_{n, i}$ with $\# I_n <\infty$ and
\[
\diam\cY_{n, i} <\ep_n \quad\text{ and }\quad
\cY_{n+1, j} \cap \cY_{n, i}=\emptyset \text{ or } \cY_{n+1, j}
\]
for any $i\in I_n$ and $j\in I_{n+1}$.
Put $I:= \{(n, i) : n\in\N, i\in I_n\}$.

Choose $Y_{n, i} \in\cY_{n, i}$ for each $(n, i)\in I$ so that $Y_{n+1, j} =Y_{n, i}$ if $Y_{n,i} \in \cY_{n+1, j}$.
Let $J:=[0, 1]\subset\R$.
If $\cY_{n+1, j} \subset\cY_{n, i}$,
then $\Box(Y_{n, i}, Y_{n+1, j}) <\ep_n$
and hence there exist parameters $\phi:=\phi_{n, i, j} :J\to Y_{n, i}$ and $\psi:=\psi_{n, i, j} :J\to Y_{n+1, j}$
and a Borel set $J':=J_{n, i, j} \subset J$ with $\cL(J')>1-\ep_n$ and
\[
|d_{Y_{n, i}} (\phi(s), \phi(t)) - d_{Y_{n+1, j}} (\psi(s), \psi(t))| <\ep_n
\]
for any $s, t\in J'$.

Let $Z:= \bigsqcup_{(n, i)\in I} Y_{n, i}$.
For $z, z' \in Z$, we define $d_Z (z, z')$ as follows:

If $z, z' \in Y_{n, i}$ for some $(n, i)\in I$, then $d_Z (z, z'):= d_{Y_{n, i}} (z, z')$.

If $z\in Y_{n, i}$ and $z' \in Y_{n+1, j}$ for some $(n, i)\in I$ and $j\in I_{n+1}$ with $\cY_{n+1, j} \subset\cY_{n, i}$, then
\[
d_Z (z, z'):=\inf_{s\in J_{n, i, j}} [d_{Y_{n, i}} (z, \phi_{n, i, j}(s)) +d_{Y_{n+1, j}} (\psi_{n, i, j}(s), z') ] +\ep_n.
\]

If $z\in Y_{m, i}$ and $z' \in Y_{n, j}$ for some $(m, i), (n, j)\in I$ with
\[
\cY_{m, i} \subset\cY_{m-1, i(m-1)} \subset\dots\subset\cY_{l, k}, \qquad
\cY_{n, j} \subset\cY_{n-1, j(n-1)} \subset\dots\subset\cY_{l, k},
\]
and $\cY_{l+1, i(l+1)} \cap \cY_{l+1, j(l+1)} =\emptyset$, then
\[
d_Z (z, z'):= \inf[ d_Z (z, z_{m-1}) +\dots+ d_Z (z_{l+1}, z_l) + d_Z (z_l, z'_{l+1}) +\dots+ d_Z (z'_{n-1}, z')  ],
\]
where the infimum is taken for all
\[
(z_{m-1}, \dots, z_l, z'_{l+1}, \dots, z'_{n-1}) \in
Y_{m-1, i(m-1)}\times\dots\times Y_{l, k} \times Y_{l+1, j(l+1)} \times\dots\times Y_{n-1, j(n-1)}.
\]
Then $(Z, d_Z)$ is a separable metric space
and we replace $(Z, d_Z)$ with its metric completion if necessary.

By construction,
there exists an isometric embedding $f_{n, i} :Y_{n, i}\to Z$
and put $\mu_{n, i} := (f_{n, i})_* \mu_{Y_{n, i}} \in\cP(Z)$ for each $(n, i)\in I$.
If $\cY_{n+1, j} \subset\cY_{n, i}$, we have
\begin{align}\label{ineq;BP}
\proh(\mu_{n, i}, \mu_{n+1, j}) < 2\ep_n,
\end{align}
cf.~\cite[Proof of Lemma~4.13]{S}.

Let $Y\in\cY$.
Then there exists a unique sequence $\seqn{i(n)}$ with $i(n)\in I_n$ and $Y\in\cY_{n, i(n)}$ for any $n$.
Hence $\Box(Y_{n, i(n)}, Y)\to 0$ as $\nti$.
Inequality~\eqref{ineq;BP} implies that $\seqn{\mu_n}$ with $\mu_n :=\mu_{n, i(n)}$ is Cauchy and hence converges weakly to some $\mu\in\cP(Z)$.
Since $Y_{n, i(n)}$ and $(Z, \mu_n)$ are mm-isomorphic and
\[
\Box((Z, \mu_n), (Z, \mu)) \le 2\proh(\mu_n, \mu)\to 0
\]
as $\nti$ by Lemma~\ref{lem;BP},
$Y$ and $(Z, \mu)$ are mm-isomorphic and hence there exists an isometric embedding $f_Y: Y\to Z$ with $(f_Y)_* \muy=\mu$.

If $\seqk{\mu_k}$ is a sequence in $\cM':=\{ \mu_{n, i} : (n, i)\in I \}$,
then $\mu_k = \mu_{n(k), i(k)}$ for some $(n(k), i(k))\in I$
and there exists a subsequence which is still denoted by $\seqk{\mu_k}$ with
\[
\cY_{n(k+1), i(k+1)} \subset\cY_{n(k), i(k)} \quad\text{ and hence }\quad
\sum_{k=1}^\infty \proh(\mu_k, \mu_{k+1})<2\sum_{k=1}^\infty \ep_k <\infty
\]
by Inequality~\eqref{ineq;BP}.
Thus $\seqk{\mu_k}$ is Cauchy and hence converges. This means that $\cM'$ is precompact in $\cP(Z)$.
Since $\cM:=\{ (f_Y)_* \muy : Y\in\cY \}\subset\cP(Z)$ is contained in the closure of $\cM'$,
it is also precompact. Therefore Prohorov's theorem states that $\cM$ is tight and
the proof of (1)$\implies$(2) in Proposition~\ref{prop;embed} is completed.
\end{proof}

We use the following to prove Corollary~\ref{cor;pmG}.
\begin{lem}\label{lem;countableembed}
Let $\cY\subset\cX$ be a countable $\Box$-precompact subset.
Then there exists a complete separable metric space $Z$ for which
any $Y\in\cY$ admits an isometric embedding $f_Y :Y\to Z$
so that the map $\cY\ni Y\mapsto (f_Y)_* \muy \in\cP(Z)$ is a homeomorphism.
\end{lem}

\begin{proof}
We use the notations in our proof of Proposition~\ref{prop;embed}.
Since $(\cY, \Box)$ is a countable metric space,
we can take each $\cY_{n, i} \subset\cY$ with $\partial\cY_{n, i} = \emptyset$ so that $\cY_{n, i}$ is an open set in our proof of Proposition~\ref{prop;embed}.

Let $\seqk{Y_k}$ be a sequence in $\cY$.
We shall show that $\seqk{Y_k}$ $\Box$-converges to $Y\in\cY$ if and only if $\seqk{(f_{Y_k})_* \mu_{Y_k}}$ converges weakly to $(f_Y)_* \muy$.
The ``if" part follows from Lemma~\ref{lem;BP}.

If $\seqk{Y_k}$ $\Box$-converges to $Y\in\cY$,
there exist uniquely determined sequences $\seqn{i(k, n)}$ and $\seqn{i(n)}$ with $i(k, n), i(n)\in I_n$, $Y_k \in\cY_{n, i(k, n)}$, and $Y\in\cY_{n, i(n)}$ for any $k, n$.
Then $\seqn{\mu_{n, i(k, n)}}$ and $\seqn{\mu_{n, i(n)}}$ converge weakly to $(f_{Y_k})_* \mu_{Y_k}$ and $(f_Y)_* \muy$ respectively by construction in our proof of Proposition~\ref{prop;embed}.
For each $n$, since $\cY_{n, i(n)}$ is open, we have $Y_k\in\cY_{n, i(n)}$ and hence $i(k, n)=i(n)$ for any sufficiently large $k$.
Thus, by Inequality~\eqref{ineq;BP},
\begin{align*}
\proh((f_{Y_k})_* \mu_{Y_k}, (f_Y)_* \muy)
&=\lim_{l\to\infty} \proh(\mu_{l, i(k, l)}, \mu_{l, i(l)}) \\
&\le \sum_{l=n}^\infty [ \proh(\mu_{l, i(k, l)}, \mu_{l+1, i(k, l+1)}) + \proh(\mu_{l, i(l)}, \mu_{l+1, i(l+1)}) ]
\le 4\sum_{l=n}^\infty \ep_l,
\end{align*}
which as $n\to\infty$ implies
\[
\lim_{k\to\infty}\proh((f_{Y_k})_* \mu_{Y_k}, (f_Y)_* \muy)=0.
\]
Therefore we obtain the conclusion.
\end{proof}

\begin{proof}[Proof of Corollary~\ref{cor;pmG}]
The implication (2)$\implies$(1) follows from Lemma~\ref{lem;BP} and the implication (3)$\implies$(2) is trivial.

To prove the implication (1)$\implies$(3),
we assume that $\seqn{X_n}$ $\Box$-converges to $X$. Then, by Lemma~\ref{lem;countableembed}, there exist a complete separable metric space $(Z, d_Z)$ and isometric embeddings $f_n:X_n\to Z$ and $f:X\to Z$ for which $\seqn{(f_n)_* \mu_{X_n}}$ converges weakly to $f_*\mux$.
Moreover, given a point $x\in\supp\mux$, there exist points $x_n \in\supp\mu_{X_n}$ with $d_Z (f_n (x_n), f(x))\to 0$,
because
\[
\liminf_\nti \mu_{X_n} (f_n^{-1}(O)) \ge f_* \mux(O) =\mux(f^{-1}(O)) >0
\]
for any open set $O\subset Z$ with $f(x)\in O$ by the Portmanteau theorem.
Then $\seqn{(X_n, x_n)}$ pmG-converges to $(X, x)$. This completes the proof.
\end{proof}

\subsection*{Acknowledgements}
The first author was partly supported by JSPS KAKENHI (No.20J00147)
and the second author was partly supported by JSPS KAKENHI (No.18K03298).

\end{document}